\documentclass[11pt]{amsart}
\usepackage{amsmath,amssymb,amsfonts,eucal,mathrsfs}
\usepackage{amsthm}
\usepackage{bm}
\usepackage{bbm}
\usepackage{geometry}
\usepackage{array}
\usepackage{resizegather}
\usepackage{graphicx}
\usepackage{color}
\usepackage[normalem]{ulem}
\usepackage[latin1]{inputenc}
\usepackage{mathtools}

\theoremstyle{plain}
\newtheorem{thm}{Theorem}[section]
\newtheorem{thmx}{Theorem}

\newtheorem{corollary}[thm]{Corollary}

\theoremstyle{definition}

\newtheorem{conjecture}{Conjecture}

\newtheorem{remark}{Remark}[section]

\numberwithin{equation}{section}

\newcommand{\bo}{{\rm O}}
\newcommand{\ds}{\displaystyle}

\newcommand{\dsum}{\ds\sum}

\newcommand{\eqskip}{ \vspace*{2mm}\\ }
\newcommand{\R}{\mathbb{R}}
\newcommand{\N}{\mathbb{N}}

\newcommand{\so}{{\rm o}}

\newcommand{\fr}[2]{\frac{\ds #1}{\ds #2}}

\newcommand{\txtb}{\textcolor{blue}}

\makeatletter
\@namedef{subjclassname@2020}{%
  \textup{2020} Mathematics Subject Classification}
\makeatother

\begin{document}

\title{On the (growing) gap between Dirichlet and Neumann eigenvalues}

\author[P. Freitas]{Pedro Freitas}
\author[M. Gama]{Miguel Gama}
\address{Grupo de F\'{\i}sica Matem\'{a}tica, Instituto Superior T\'{e}cnico, Universidade de Lisboa, Av. Rovisco Pais, 1049-001 Lisboa, Portugal}
\email{pedrodefreitas@tecnico.ulisboa.pt,\, miguel.gama@tecnico.ulisboa.pt}

\date{\today}

\begin{abstract}
We provide an answer to a question raised by Levine and Weinberger in their $1986$ paper concerning the
difference between Dirichlet and Neumann eigenvalues of the Laplacian on bounded domains in $\R^{n}$. More precisely,
we show that for a certain class of domains there exists a sequence $p(k)$ such that
$\lambda_{k}\geq \mu_{k+ p(k)}$ for sufficiently large $k$. This sequence, which is given explicitly
and is independent of the domain, grows with $k^{1-1/n}$ as $k$ goes to infinity, which we conjecture to be optimal.
We also prove the existence of a sequence, now not given explicitly and
only of order $k^{1-3/n}$ but valid for bounded Lipschitz domains in $\R^{n} (n\geq4)$, for which a similar
inequality holds for all $k$. We then frame these general results with some specific planar Euclidean examples such
as rectangles and disks, for which we provide bounds valid for all eigenvalue orders.
\end{abstract}

\maketitle
%\twocolumn

\section{Introduction}
We consider the Dirichlet and Neumann eigenvalue problems for the Laplacian on a bounded domain $\Omega$
in $\R^{n}$ defined by
\begin{equation}\label{dir}
\left\{
\begin{array}{rl}
  \Delta u + \lambda u = 0, & x\in\Omega\eqskip
  u = 0, & x\in\partial \Omega
 \end{array}
\right.
\end{equation}
and
\begin{equation}\label{neu}
\left\{
\begin{array}{rl}
  \Delta v + \mu v = 0, & x\in\Omega\eqskip
  \fr{\partial v}{\partial \nu} = 0, & x\in\partial \Omega,
 \end{array}
\right.
\end{equation}
respectively, where $\nu$ denotes the outer unit normal on $\partial\Omega$, the boundary of $\Omega$. We shall denote the corresponding spectra by
$\Sigma_{D}$ and $\Sigma_{N}$, respectively, and write the eigenvalues as
\[
 0<\lambda_{1}\leq \lambda_{2} \leq \dots
\]
and
\[
 0=\mu_{1}\leq \mu_{2} \leq \dots.
\]
The study of inequalities of the type $\lambda_{k}\geq \mu_{k+m}$ for all $k$ and some fixed $m$ dates at least as far back as the work of Payne in $1955$, who showed that $\lambda_{k}\geq\mu_{k+2}$ for planar convex domains with a sufficiently smooth boundary~\cite{payn}. It took about thirty years for this
result to be generalised to higher dimensions by Aviles~\cite{avil} and Levine and Weinberger~\cite{levwei} in 1986. Among
other results where the curvature of the boundary plays a key role, it is shown in~\cite{levwei} that for smooth bounded convex domains in $\R^{n}$ we have
\begin{equation}\label{convineq}
 \lambda_{k}\geq\mu_{k+n}.
\end{equation}
More recently, it was shown by Rohleder that Payne's result may be extended to planar simply-connected domains~\cite{rohl}.

The other main development in this direction was made by Friedlander in $1991$, who proved a conjecture
of Payne's (see~\cite{levi,payn2}), namely, that
\begin{equation}\label{paynconj}
 \lambda_{k} > \mu_{k+1}
\end{equation}
for all bounded sufficiently smooth domains $\Omega$ in $\R^{n}$~\cite{frie} -- a different proof that extended this result to
domains for which an embedding condition is satisfied was later given by Filonov~\cite{filo}; see also Remark 1.9 in~\cite{safa},
observing that Filonov's proof holds for general domains. Note also that, in the case of tiling domains,~\eqref{paynconj} is an immediate consequence of P\'{o}lya's inequalities for Dirichlet and Neumann eigenvalues.

In this paper we shall consider two aspects related to inequalities between Dirichlet and Neumann
eigenvalues. The first of these is motivated by a question posed at the end of~\cite{levwei}, asking whether
inequality~\eqref{convineq} may {\it be replaced by a better inequality of the form
\begin{equation}\label{gapk}
  \lambda_{k} > \mu_{\phi(n,k)}
\end{equation}
for convex $n-$dimensional domains.} -- see also the comments in the second to last paragraph on page $44$
of~\cite{mazz}, referring to the behaviour for large $k$. To address this question, we consider the two-term Weyl
asymptotics for eigenvalues of problems~\eqref{dir} and~\eqref{neu}, namely,
\begin{equation}\label{weyl}
 \lambda_{k} = c_{0}k^{2/n} + c_{1} k^{1/n} + \so\left(k^{1/n}\right)
\end{equation}
and
\[
 \mu_{k} = c_{0}k^{2/n} - c_{1} k^{1/n} + \so\left(k^{1/n}\right),
\]
as $k\to\infty$. Here
\[
 \begin{array}{lcl}
  c_{0} = \fr{4\pi^{2}}{\left(\omega_{n}\left|\Omega\right|\right)^{2/n}} & \mbox{ and } &
  c_{1} = \fr{2\pi^{2} \omega_{n-1}\left|\partial \Omega\right|}{n\left( \omega_{n}\left|\Omega\right|\right)^{1+1/n}}
 \end{array}
\]
where $\omega_{n}$ denotes the volume of the unit ball in $\R^{n}$ and, with a slight abuse of notation, $|\Omega|$ and $|\partial\Omega|$
denote the $n-$ and $(n-1)-$volume of $\Omega$ and $\partial\Omega$, respectively. The above asymptotics hold under
a {\it non-periodicity condition} on the set of the billiard orbits defined on $\Omega$, namely, that the
set of such orbits which are periodic has measure zero -- see~\cite{sava} for precise definitions and statements. This
yields that for domains for which these two-term asymptotics are valid, such as convex domains with an analytic boundary or convex
polyhedra, the difference between Dirichlet and Neumann eigenvalues satisfies
\[
 \lambda_{k}-\mu_{k} = 2c_{1}k^{1/n} + \so\left(k^{1/n}\right),
\]
growing to infinity with $k$ and thus suggesting that it might be possible to determine an increasing sequence of
natural numbers $p=p(k)$ for which $\lambda_{k}\geq\mu_{k+p(k)}$. As far as we are aware, all existing results for differences between
Dirichlet and Neumann eigenvalues pertain to a fixed gap between the corresponding Dirichlet and Neumann indexes. One of the aims of this
paper is thus to provide a first answer to Levine and Weinberger's question and show that {\it better inequalities} of the form~\eqref{gapk}
are indeed possible in dimensions two and higher. More precisely, we shall prove that there exists an index function $\phi(n,k)=k+p(k)$ with the
sequence $p(k)$ of order $k^{1-1/n}$ and independent of the domain, such that $\lambda_{k}\geq\mu_{k+p(k)}$ for all
sufficiently large values of $k$.
\begin{thmx}\label{thmpasympt}
 Let $\Omega\subset\R^n$ be a bounded domain satisfying the non-periodicity condition and which is not a ball. Then there
 exists $k^{*} = k^{*}(\Omega)$ such that
 \[
  \lambda_{k}(\Omega) \geq \mu_{k+p(k)}(\Omega), \mbox{ for all } k\geq k^{*}.
 \]
 where $p(k) = \left\lfloor\fr{n \omega_{n-1}}{2\omega_{n}^{1-2/n}} k^{1-1/n}\right\rfloor$.
\end{thmx}
\begin{remark} Concerning the exclusion of the ball from the statement, see the discussion following the proof of the theorem in Section~\ref{proofs},
as well as the specific results for disks given in Section~\ref{disk_sec}.
\end{remark}

For general Lipschitz domains and in dimensions higher than three, we show that there exists a sequence $p(k)$
such that $\lambda_{k}\geq\mu_{k+p(k)}$ for all $k$, as a direct consequence of a result of Safarov and Filonov's
for the difference between the counting functions for Dirichlet and Neumann eigenvalues. However, now $p$ depends on an
unspecified constant and its asymptotic behaviour is slightly worse than that of the sequence above -- see
Theorem~\ref{thmpasympt2} and Remark~\ref{remasympt} for a discussion of the two results.

The second aspect we consider concerning this type of inequality turns up as a consequence of the proof of Theorem~\ref{thmpasympt}.
Together with the above two-term asymptotic expansions, it turns out that another key ingredient appearing in
the proof is the Euclidean geometric isoperimetric inequality, namely,
\begin{equation}\label{geomisosp}
\left|\partial \Omega\right| \geq n  \left|\Omega\right|^{1-1/n}\omega_{n}^{1/n}.
\end{equation}
This points in the direction that convexity might not be a crucial condition, or that at least it might be possible
to hope for a result of this type for all bounded domains. In fact, we see that the isoperimetric inequality plays
a role in determining the form of the asymptotic behaviour of $p$ as $k$ goes to infinity -- see
Remark~\ref{isoperim_manif}. This connection was also noted in~\cite{cms} with respect to the number
of Neumann eigenvalues which are smaller than or equal to the first Dirichlet eigenvalue. We note further that the
proof of the results for rectangles given in Section~\ref{sec_rect}, now valid for all $k$, also makes the corresponding
relation between the perimeter and area of a rectangle appear explicitly. In fact one of the results we obtain is of the form
\[
  \lambda_{k} \geq \mu_{k + \left\lfloor P \sqrt{\fr{k}{\pi A} }\right\rfloor} \;\; \mbox{ for all } k\in\N,
\]
where $\left\lfloor \cdot \right\rfloor$ denotes the floor function, with $P$ and $A$ the perimeter and area of the rectangle,
respectively. This points in the direction that the larger the isoperimetric constant associated with the rectangle, the larger
the gap between the indexes of the corresponding Dirichlet and Neumann eigenvalues appearing in the inequalities, suggesting
that such a result could hold for other domains. However, the examples given in~\cite{hatc} show that this will not necessarily
be the case when the domains are allowed to be nonconvex, while our own results for the disk provide a counterexample at the
low end of the values for the isoperimetric deficit -- see the remarks following Conjecture~\ref{alldomains}.

It is possible to do variations on these inequalities, either emphasising the dependence on the side-lengths of the rectangle
or on obtaining inequalities which are independent of these -- see Section~\ref{sec_rect} for other results.

On the other hand, it is known that inequalities of this type between Dirichlet and Neumann eigenvalues in
some non-Euclidean settings such as spheres, may not hold and may, in fact, be reversed -- see the results and discussions
in~\cite{ashlev},~\cite{levi} and~\cite{mazz}; in particular, the second and third papers both refer back to a result that may
be found in~\cite{chav}, namely, that for geodesic disks on $\mathbb{S}^{n}$ whose radius is strictly between $\pi/2$ and $\pi$,
we have $\lambda_{1}<\mu_{2}$. Note that for these disks the non-periodicity
condition is not satisfied, again suggesting this to be a relevant condition. We also have that the isoperimetric inequality
satisfied by domains on $\mathbb{S}^{n}$ is, of course, not the same as that in $n-$Euclidean space but, as we
shall see in Section~\ref{sec-sphere}, this does not, by itself, perclude inequalities analogous to those
in Euclidean space to hold. The form of the isoperimetric inequaltiy does, however, control what is possible and
to what extent.

Based on the combination of results obtained we believe that the two key ingredients mentioned above and which appear in the
proof of Theorem~\ref{thmpasympt} are decisive factors for a result of this type to hold. As such, we formulate the following
conjecture for general Euclidean bounded domains -- note that the non-periodicity condition is conjectured to hold for general
Euclidean domains.
\begin{conjecture}\label{alldomains}
There exist constants $C_{n}$ in $(0,1]$, depending only on the dimension, such that
the Dirichlet and Neumann eigenvalues of a bounded domain $\Omega$ in $\R^{n}$ satisfy the inequality
 \[
  \lambda_{k}(\Omega) \geq \mu_{k+p(k)}(\Omega)
 \]
for all positive integer $k$, where $p(k) = \left\lfloor C_{n}\fr{n \omega_{n-1}}{2\omega_{n}^{1-2/n}} k^{1-1/n}\right\rfloor$.
Furthermore, the power $1-1/n$ is optimal.
\end{conjecture}
\noindent In essence, this is stating that although the power given in Theorem~\ref{thmpasympt} cannot be improved, the constant
multiplying the term $k^{1-1/n}$ in that theorem, although correct asymptotically, might be too large for the result to hold for
all $k$ and a factor has to be introduced. This follows from the results for the disk given in Section~\ref{disk_sec}, from which we see
the result given in Theorem~\ref{thmpasympt} cannot hold for certain small values of $k$, at least up to and including dimension four.

\begin{remark}
 As has been pointed out in~\cite{blp}, the claim made in~\cite{levwei} that for the annular sector given by
 \[
  S = \left\{ (r,\theta): 1<r<2 \wedge 0< \theta < 3\pi/2\right\}
 \]
 we have $\mu_{3}>\lambda_{1}$ is incorrect. % -- note that if this were not the case, it would immediately prove
%  both conjectures wrong.
 As is well known, both the Dirichlet and Neumann eigenvalues of $S$ may be obtained by separation of variables
 and then solving the resulting equations involving the Bessel functions of the first and second kind, $J_{\nu}$ and $Y_{\nu}$,
 respectively, $J_{2k/3}$, $J_{2k/3\pm 1}$, $Y_{2k/3}$ and $Y_{2k/3\pm 1}$, $k\in\N_{0}$. The first Dirichlet eigenvalue is given
 by $\lambda_{1}(S) \approx 9.96001$, while the corresponding Neumann eigenvalues are (approximately) given by
 \[
  \left\{0,\, 0.204718,\, 0.811126,\, 1.79721,\, 3.13054,\, 4.77455,\, 6.69575,\, 8.86914,\, 10.2181,\, 10.4649\right\}.
 \]
 We thus see the first Neumann eigenvalue to be larger than $\lambda_{1}(S)$ is $\mu_{9}(S)$. We further
 note that this behaviour is not very different from what one has for a rectangle with side lengths $1$ and $9\pi/4$, corresponding
 to a rectangle with the same perimeter and area as the annular sector. In this case $\lambda_{1}= \pi^2 + 16/81\approx 10.0671$, and
 the first ten Neumann eigenvalues are (approximately) given by
 \[
 \left\{ 0., 0.197531, 0.790123, 1.77778, 3.16049, 4.93827, 7.11111, 9.67901, 9.8696, 10.0671\right\},
 \]
 showing that $\mu_{10}$ is now the first Neumann eigenvalue to equal $\lambda_{1}$.
\end{remark}

Inspired by the result for rectangles in Theorem~\ref{genrect} already mentioned above, and by the results in
Section~\ref{generalplanar} for general planar domains, it is also possible to formulate other conjectures with a sequence $p$
depending on the isoperimetric
constant of the given domain. Again, this cannot hold either in general or even for convex domains, as both the examples given
in~\cite{hatc} for nonconvex domains, and those for the disk in Section~\ref{disk_sec} below show. We thus restrict the conjecture
to convex domains, and, for simplicity, state it only in the planar case.
\begin{conjecture}\label{planarconj}
There exists a constant $C$ in $(0,1$) such that the Dirichlet and Neumann eigenvalues of a bounded convex planar domain $\Omega$
with perimeter $P$ and area $A$ satisfy the inequality
 \[
  \lambda_{k}(\Omega) \geq \mu_{k + \left\lfloor C P \sqrt{\fr{k}{\pi A} }\right\rfloor}(\Omega) \;\; \mbox{ for all } k\in\N.
 \]
\end{conjecture}

\section{General results for Euclidean domains\label{proofs}}

In this section we address what might be considered as proof-of-concept results showing how we should expect the difference between
the Dirichlet and Neumann eigenvalues to grow. We begin with the proof of Theorem~\ref{thmpasympt}, which requires the non-periodicity
condition to hold, and is valid only for sufficiently large enough values of the indexes.

\begin{proof}[Proof of Theorem~\ref{thmpasympt}] For simplicity of notation, we shall drop the explicit dependence of $p$ on $k$.
 From
 \[
  \lambda_{k} = c_{0} k^{2/n} + c_{1} k^{1/n} + r_{D}(k)
 \]
 and
 \[
  \mu_{k} = c_{0} k^{2/n} - c_{1} k^{1/n} + r_{N}(k)
 \]
 we obtain
 \[
 \begin{array}{lll}
  \mu_{k+p} & = & c_{0}(k+p)^{2/n} - c_{1}(k+p)^{1/n}+r_{N}(k+p)\eqskip
  & = & c_{0}k^{2/n}\left(1+\fr{p}{k}\right)^{2/n} - c_{1}k^{1/n}\left(1+\fr{p}{k}\right)^{1/n}+r_{N}(k+p)\eqskip
  & = & c_{0}k^{2/n} + c_{1}k^{1/n} + r_{D}(k)\eqskip
  & & \hspace*{5mm} + c_{0}k^{2/n}\left[\left(1+\fr{p}{k}\right)^{2/n}-1\right]
  -c_{1}k^{1/n}\left[1+\left(1+\fr{p}{k}\right)^{1/n}\right] + r_{N}(k+p)-r_{D}(k)\eqskip
  & = & \lambda_{k} + c_{0}k^{2/n}\left[\left(1+\fr{p}{k}\right)^{2/n}-1\right]
  -c_{1}k^{1/n}\left[1+\left(1+\fr{p}{k}\right)^{1/n}\right] + r_{N}(k+p)-r_{D}(k).
 \end{array}
 \]
Hence
\[
 \lambda_{k} - \mu_{k+p} = r_{D}(k)-r_{N}(k+p) + \underbrace{c_{1}k^{1/n}\left[1+\left(1+\fr{p}{k}\right)^{1/n}\right]
 -c_{0}k^{2/n}\left[\left(1+\fr{p}{k}\right)^{2/n}-1\right]}_{\text{g(n,k,p)}}.
\]
The remainder of the proof is divided into two parts.
We shall first derive a condition for the term $g(n,k,p)$ on the right to be positive, and then show that
this is satisfied by the expression for $p$ given above. We then show that this gives a term of order $k^{1/n}$, and
is thus larger than the difference $r_{D}(k)-r_{N}(k+p)=\so(k^{1/n})$ for sufficiently large $k$.

Letting $x=(k+p)^{1/n}$ we may write $g$ as
\begin{align}
%   \begin{array}{lll}
  g(n,k,p) & =  c_{1}\left( k^{1/n} + x\right) -c_{0}\left(x^2-k^{2/n}\right)\nonumber\eqskip
  & = -c_{0}x^{2} +c_{1}x +\left(c_{0}k^{2/n}+c_{1} k^{1/n} \right)\label{gx}.
%   \end{array}
\end{align}
This will be positive if
\[
\begin{array}{lll}
 \left( 0 \leq\right) x & \leq & \fr{c_{1} + \sqrt{c_{1}^2+4k^{1/n}\left( c_{0}k^{1/n}+c_{1}\right)c_{0}}}{2c_{0}}\eqskip
 & = & \fr{c_{1}}{c_{0}}+k^{1/n}.
\end{array}
\]
We thus have that $p$ must satisfy
\[
\begin{array}{lll}
 p & \leq & \left( \fr{c_{1}}{c_{0}}+k^{1/n}\right)^{n}-k\eqskip
 & = & \left[\fr{\omega_{n-1} \left|\partial\Omega\right|}{2n\left( \omega_{n}\left|\Omega\right|\right)^{1-1/n}}+k^{1/n}\right]^{n}-k
\end{array}
\]
From the Euclidean geometric isoperimetric inequality~\eqref{geomisosp} we have that the right-hand side above satisfies
\[
 \begin{array}{lll}
  \left[\fr{\omega_{n-1} \left|\partial\Omega\right|}{2n\left( \omega_{n}\left|\Omega\right|\right)^{1-1/n}}+k^{1/n}\right]^{n}-k &
  \geq & \left( \fr{\omega_{n-1}}{2\omega_{n}^{1-1/(2n)}}+k^{1/n} \right)^{n}-k\eqskip
  & = & k \left[\left(1+ \fr{\omega_{n-1}}{2\omega_{n}^{1-2/n}}\times \fr{1}{k^{1/n}}\right)^{n}-1\right]\eqskip
  & > & k \left(1+ \fr{n\omega_{n-1}}{2\omega_{n}^{1-2/n}}\times \fr{1}{k^{1/n}}-1\right)\eqskip
  & = & \fr{n\omega_{n-1}}{2\omega_{n}^{1-2/n}}\times {k^{1-1/n}},
 \end{array}
\]
where the strict inequality comes from applying Bernoulli's inequality with $n$ greater than or equal to two. From this it follows
that if we take $p$ to satisfy
\begin{equation}\label{popt}
 p = \left\lfloor\fr{n \omega_{n-1}}{2\omega_{n}^{1-2/n}} k^{1-1/n}\right\rfloor
\end{equation}
then $g$ is a strictly positive function.
It remains to prove that the resulting term when $p$ takes on this value is of order $k^{1/n}$. We first note that from~\eqref{gx}
it follows that $g$ will be strictly decreasing in $p$ for sufficiently large $k$. Thus, showing that this will be of order $k^{1/n}$ for a
value of $p$ larger than that given by~\eqref{popt} will imply the desired result. From
\[
 p = \left\lfloor\fr{n \omega_{n-1}}{2\omega_{n}^{1-2/n}} k^{1-1/n}\right\rfloor \leq
 \fr{n \omega_{n-1}}{2\omega_{n}^{1-2/n}} k^{1-1/n} \leq n\fr{c_{1}}{c_{0}}k^{1-1/n}
\]
we see that it is enough to prove the asymptotic behaviour for $g\left(n,k,n\fr{c_{1}}{c_{0}}k^{1-1/n}\right)$.
Writing
\[
   g\left(n,k,n\fr{c_{1}}{c_{0}}k^{1-1/n}\right)  =  c_{1} k^{1/n}\left[1+\left(1+\fr{\alpha}{k^{1/n}}\right)^{1/n}\right]
   -c_{0}k^{2/n}\left[\left(1+\fr{\alpha}{k^{1/n}}\right)^{2/n}-1\right]
\]
with $\alpha = \fr{n\omega_{n-1}}{2\omega_{n}^{1-1/n}}$, we then have
\[
 \begin{array}{lll}
  g\left(n,k,n\fr{c_{1}}{c_{0}}k^{1-1/n}\right) & \approx & k^{1/n} \left[ c_{1} + c_{1}\left( 1+ \fr{\alpha}{n k^{1/n}}+ \dots\right) - c_{0}k^{1/n}\left( 1+\fr{2\alpha}{n k^{1/n}}+ \dots -1 \right) \right]\eqskip
  & = & k^{1/n} \left( 2c_{1} + \fr{\alpha}{n k^{1/n}} - 2\fr{\alpha c_{0}}{n} + \dots\right)\eqskip
  & = & 2\left( c_{1}-\fr{\alpha c_{0}}{n}\right) k^{1/n} + \bo(1),
 \end{array}
\]
as $k$ goes to infinity.
Since $\alpha < \fr{n c_{1}}{c_{0}}$, provided $\Omega$ is not a ball, the coefficient affecting the leading term
$k^{1/n}$ is strictly positive, proving the result.
\end{proof}
\begin{remark}\label{isoperim_manif}
 We believe the exclusion of the ball from the result to be a purely technical matter, at least in dimensions larger than two.
 In this case, however, it is straightforward to carry out all of the calculations in the proof of Theorem~\ref{thmpasympt}
 (without using Bernoulli's inequality) to obtain, with $p(k) = \lfloor 2\sqrt{k}\rfloor + 1$,
 \[
\begin{array}{lll}
\lambda_{k} - \mu_{k+p} & =  & r_{D}(k)-r_{N}(k+p) \vspace*{3mm}\\
& &  \hspace*{0.5cm} + \fr{2\pi}{ A} \underbrace{\left[
 \left( \sqrt{k+\lfloor 2 \sqrt{k} \rfloor + 1}+\sqrt{k}\right)
\frac{P}{\sqrt{\pi A}}-2\left(\lfloor  2 \sqrt{k} \rfloor + 1\right)\right]} _{\text{g(k)}}.
\end{array}
 \]
 We now see that the case of the disk is indeed special, in that it is the only planar domain for which the function $g(k)$
 above does not dominate the difference of the remainder terms and it is, in fact, bounded -- see also
 Section~\ref{disk_sec}. In any case, this also stresses the connection to the isoperimetric inequality, which appears naturally
 in the proof as a result of the combination of the two coefficients $c_{0}$ and $c_{1}$ in the two-term Weyl asymptotic formula~\eqref{weyl}. Note also that since these coefficients are the same for the problem on manifolds, provided~\eqref{weyl}
 is satisfied, this hints at why we cannot expect the same result to hold on manifolds.
\end{remark}

\begin{remark}
 By writing the function $g$ in the proof of the theorem as
 \[
  g(n,k,p) = c_{0}\left[k^{1/n} + (k+p)^{1/n}\right] \left[ k^{1/n} + \fr{c_{1}}{c_{0}}-(k+p)^{1/n} \right],
 \]
 we see that, under the natural assumption that $p(k) = \so(k)$, the term inside the second pair of square brackets
 above behaves asymptotically as $-p(k) k^{1/n-1}/n$. Thus, unless $p = \bo\left( k^{1-1/n} \right)$, this
 term goes to minus infinity as $k$ grows, yielding that, at least when the two-term Weyl asymptotic formula~\eqref{weyl}
 holds, the asymptotic behaviour of order $k^{1-1/n}$ for $p$ is optimal.
\end{remark}

\begin{remark} We have
\[
\fr{n \omega_{n-1}}{2\omega_{n}^{1-2/n}} =
\frac {\sqrt {\pi}  4^{\frac {1} {n} - 1}  n\left (n\Gamma\left (\frac {n} {2} \right) \right)^{\frac {n - 2} {n}}} {\Gamma\left (\frac {n + 1} {2} \right)}\approx e\sqrt{\fr{\pi}{2}n}+ \so(\sqrt{n}) \mbox{ as } n\to \infty.
\]
This means that for large $n$ we cannot expect this sequence $p$ to be optimal when $k$ is one, at least for
convex domains for which we know that~\eqref{convineq} holds.
\end{remark}
\begin{remark}
 Since the asymptotic behaviour of eigenvalues of the Laplace operator with Robin boundary conditions follows
 the same behaviour as the two-term asymptotics for the Neumann problem, the above result also holds for Robin
 boundary conditions of the form $\partial u/\partial \nu + \beta u = 0$ with positive $\beta$. However, in this
 case and since for any given integer $m$, by making $\beta$ large enough, we can make the first $m$ eigenvalues
 of the Robin spectrum as close to the corresponding first $m$ Dirichlet eigenvalues as we want, we cannot expect
 any such set of inequalities to be valid for all $k$ without imposing any further restrictions. When
 $\beta$ is allowed to be negative, then it was shown in~\cite{gemi} that Friedlander's
 inequalities~\eqref{paynconj} continue to hold, and that this also extends to more general (nonlocal) Robin
 boundary conditions.
\end{remark}

Using the result of Safarov and Filonov's for the difference $N_{N}(\lambda)-N_{D}(\lambda)$ mentioned in the
Introduction~\cite{saffil}, we may prove a result for all $k$ and general Lipschitz domains. The price to
pay is that this does not provide an explicit constant and the asymptotic growth in $k$ is weaker.

\begin{thm}\label{thmpasympt2}
 Let $\Omega$ be a bounded domain in $\R^{n}$ ($n\geq 4$) with a Lipschitz boundary. Then there exists a positive
 constant $C_{\Omega}$ such that
 \[
  \lambda_{k} \geq \mu_{k + \left\lfloor C_{\Omega} k^{1-3/n}\right\rfloor}
 \]
 for all positive integer $k$.
\end{thm}
\begin{remark}\label{remasympt}
 This result has the advantage that it does hold for all $k$ and general Lipschitz domains. On the other hand, and
 apart from requiring $n$ greater than or equal to four to provide relevant information, the constant $C_{\Omega}$
 is not explicit and the exponent $1-3/n$ is smaller than the corresponding $1-1/n$ exponent in Theorem~\ref{thmpasympt}.
 The latter exponent corresponds to the conjecture mentioned in Remark~4.3 in~\cite{saffil}.
\end{remark}
\begin{proof}
 Define the Dirichlet and Neumann counting functions by
 \begin{equation}\label{countf}
 \begin{array}{lll}
  N_{D}(\lambda) = \#\left\{ \lambda_{k}\in\Sigma_{D}: \lambda_{k} < \lambda\right\} & \mbox{ and } &
  N_{N}(\lambda) = \#\left\{ \mu_{k}\in\Sigma_{N}: \lambda_{k} < \lambda\right\}
  \end{array}.
 \end{equation}
 We now start from a consequence of Theorem~4.1 in~\cite{saffil}, namely equation~(4.3) in that paper that states
 that for Lipschitz domains in $\R^{n}$ there exists a constant $C= C(\Omega)$ such that these functions satisfy
 \[
  N_{N}(\lambda) - N_{D}(\lambda) \geq C(\Omega) \lambda^{(n-3)/2}
 \]
 for all positive values of $\lambda$. Take $\lambda\in(\lambda_{k},\lambda_{k+1})$ for some $k$. Then $N_{D}(\lambda) = k$ and we have from~\cite{liyau} that $\lambda$ satisfies
 \[
  \lambda > \lambda_{k} \geq \fr{n}{n+2} c_{0}k^{2/n},
 \]
 where $c_{0}$ is the same constant as in the first term in the Weyl asymptotics~\eqref{weyl}. Hence
 \[
  N_{N}(\lambda) \geq k + C(\Omega) \left(\fr{n}{n+2} c_{0}\right)^{(n-3)/2} k^{(n-3)/n} = k + C_{\Omega}k^{1-3/n},
 \]
 for some constant $C_{\Omega}$. We thus conclude that
 \[
  \lambda \geq \mu_{N_{N}(\lambda)} \geq \mu_{k + C_{\Omega}k^{1-3/n}}
 \]
 and since we may take $\lambda$ arbitrarily close to $\lambda_{k}$ we obtain the desired result.
\end{proof}

\section{Two--dimensional examples} The purpose of this section is to explore further examples illustrating
the type of results that may be expected to hold for all indices.

\subsection{Rectangles\label{sec_rect}}
We begin by giving two different results for rectangles to illustrate what may (and may not) be
expected for general domains. Similar results may be obtained for higher dimensions, essentially in the
same way but with the calculations becoming more involved.

\begin{thm}\label{genrect}
 For any rectangle and all positive integer $k$ we have
 $
  \lambda_{k} \geq \mu_{k + \left\lfloor P \sqrt{\fr{k}{\pi A} }\right\rfloor}.
 $
\end{thm}
\begin{proof}
 The proof is similar to that of Theorem~\ref{thmpasympt2}. Let $R$ be a rectangle with side lengths $a$ and $b$.
 The Dirichlet and Neumann eigenvalues of $R$ are given by
 \[
  \lambda_{k} = \pi^2\left(\fr{q^{2}}{a^{2}}+ \fr{r^{2}}{b^{2}}\right), q,r\in\N
 \]
and
 \[
  \mu_{k} = \pi^2\left(\fr{q^{2}}{a^{2}}+ \fr{r^{2}}{b^{2}}\right), q,r\in\N_{0}.
 \]
 As usual we associate this with an integer lattice counting problem on the plane $qr$.
 With $N_{D}$ and $N_{N}$ the counting functions defined by~\eqref{countf} we have that the difference between
 these two functions is now given precisely by the number of points on the positive $q$ and $r$ axes plus one
 (corresponding to the zero Neumann eigenvalue).
 More precisely,
 \[
  N_{N}(\lambda) - N_{D}(\lambda) = 1 + \left\lfloor \fr{a\sqrt{\lambda}}{\pi}\right\rfloor +
  \left\lfloor\fr{b\sqrt{\lambda}}{\pi}\right\rfloor.
 \]
 We now proceed exactly as before, except that since rectangles satisfy P\'{o}lya's conjecture we have the stronger
 inequality $\lambda_{k}\geq 4\pi k/(ab)$ yielding
\begin{align}
  N_{N}(\lambda) & \geq  k + 1 + \left\lfloor 2 \sqrt{\fr{a k}{b\pi}}\right\rfloor +
  \left\lfloor 2 \sqrt{\fr{b k}{a\pi}}\right\rfloor\label{bestp}\eqskip
  & \txtb{\geq}  k + \left\lfloor 2 \left(\fr{a+b}{\sqrt{ab}}\right)\sqrt{\fr{k}{\pi}}\right\rfloor\label{bestp2}\eqskip
  & =  k + \left\lfloor P \sqrt{\fr{k}{\pi A} }\right\rfloor\nonumber.
\end{align}
 The remaining part of the proof now follows as before.
\end{proof}

Following along a similar path, it is possible to derive a result independent of the dimensions of the rectangle.
\begin{corollary}\label{genrect2}
 For any rectangle and all positive integer $k$ the corresponding eigenvalues satisfy
 \[
  \lambda_{k} \geq \mu_{k + \left\lfloor 4\sqrt{\fr{k}{\pi}}\right\rfloor}.
 \]
\end{corollary}
\begin{proof} Apply the geometric isoperimetric inequality for quadrilaterals $P^2\geq 16A$ to the previous
result.
\end{proof}
The constant $4/\sqrt{\pi}$ appearing in the inequality in Theorem~\ref{genrect2} as the coefficient of the term in
$\sqrt{k}$ is optimal for a sequence $p$ of this type, as may be seen from the asymptotic behaviour of the eigenvalues
of the square. We
conjecture that it is possible to take $p(k) = \left\lfloor 4\sqrt{\fr{k}{\pi}}\right\rfloor + 1$, but have only been
able to prove it in the case where $a\geq 4b$, where this is a direct consequence of~\eqref{bestp}. By using the
explicit expressions for the optimal values obtained for each $\lambda_{k}$ for $k=1,\dots,15$ provided in~\cite{antfre},
instead of the lower bound given by P\'{o}lya's inequality, it is possible obtain that the inequality holds with
$p(k) = \left\lfloor\left(\frac{4}{\sqrt {\pi}} - \frac{1}{\sqrt {27}}\right)\sqrt{k}\right\rfloor + 1$, and using further values will help
to improve this. However, although this will allow us to diminish the constant subtracting from $4/\sqrt{\pi}$, it will
remain positive and hence this approach does not allow for a proof of the conjecture.

\subsection{Disks\label{disk_sec}} In the case of disks, the inequality corresponding to that in Theorem~\ref{genrect} would be
\[
 \lambda_{k} \geq \mu_{k+\left\lfloor 2\sqrt{k}\right\rfloor.}
\]
We shall first see that this inequality is, in fact, not satisfied in the case of the disk which, in particular, implies that
should either of Conjectures~\ref{alldomains} or~\ref{planarconj} hold, the corresponding constants $C_{2}$ and $C$ must be smaller
than one. To do this, it is sufficient to determine up to the fifth Dirichlet eigenvalue and compare it with the corresponding
ninth Neumann eigenvalue. Denoting, as usual, by $J_{l}$ the Bessel function of the first kind of order $l$ and by $j_{l,k}$ its
$k^{\rm th}$ (positive) zero, the first five Dirichlet eigenvalues of the unit disk, in increasing
order and including multiplicities, are given by the squares of the zeros $j_{0,1}<j_{1,1}=j_{1,1}<j_{2,1}=j_{2,1}$, so that
$\lambda_{5}(D) = j_{2,1}^2 \approx 5.13562^2$. The Neumann eigenvalues of the unit disk are the squares of zeros of $J_{l}'(r)$,
usually denoted by $j_{l,k}'$. We now have $0=j_{0,1}' < j_{1,1}' = j_{1,1}' < j_{2,1}' = j_{2,1}' < j_{0,2}' < j_{3,1}' = j_{3,1}' 
< j_{4,1}'$, and $\mu_{9} = \left(j_{4,1}'\right)^2 \approx 5.31755^2$, so that $\lambda_{5} < \mu_{5 + \lfloor{2\sqrt{5}}\rfloor}$.

For $p=\lfloor2 \sqrt{k}\rfloor$ as above, this example is not an exception, and for $k$ between one and $50$, for instance,
there are nine cases where $\lambda_{k}<\mu_{k+\lfloor 2\sqrt{k}\rfloor}$, namely $k=5,8,21,27,29,34,42,49,50$. Furthermore, we see
that the first of these examples holds for $p(k) = \lfloor c \sqrt{k}\rfloor$ with $c$ down to $4/\sqrt{5}\approx 1.78885$, for which we still
obtain the Neumann index $9$ when $k$ is five.

Note that, for the disk, the function $p(k)=\lfloor2 \sqrt{k} \rfloor$ is the same as that appearing in Theorem~\ref{thmpasympt}. From
this perspective, we may also consider the ball in higher dimensions and see how it fares in terms of Conjecture~\ref{alldomains}. 
Proceeding as above for the disk, we see that while in dimensions three and four there are still examples where the statement in 
Conjecture~~\ref{alldomains} fails for the ball if $C_{n}$ is taken to be one, no such examples were found in dimensions five to eight.

We shall now determine functions $p(k)$ for which the result holds for the disk for all $k$. For clarity, we first state and prove a
result illustrating the method used, while providing a simple expression for $p$. This is then improved below by
including an arbitrary large number of terms in the sum affecting the $\sqrt{k}$ term.

\begin{thm}\label{disk1}
The Dirichlet and Neumann eigenvalues of the disk satisfy
\[
\lambda_k \geq \mu_{k+2 \left\lfloor \frac{2\sqrt{k}}{\pi} + \frac{1}{4} \right\rfloor}
\]
for all positive integer $k$.
\end{thm}
\begin{proof}
 Denote the counting functions for the Dirichlet and Neumann problems on the unit disk by
\[
\mathscr{N}_{\mathbb{D}}^D(\lambda) = \#\left\{ k : \lambda_k(\mathbb{D}) \leq \lambda^2 \right\}  \mbox{ and }
\mathscr{N}_{\mathbb{D}}^N(\lambda) = \#\left\{ k : \mu_k(\mathbb{D}) \leq \lambda^2 \right\},
\]
respectively. Following the notation in~\cite{flps}, we write
\[
\mathscr{P}_2^D(\lambda) = \left\lfloor\dfrac{\lambda}{\pi}+\dfrac{1}{4}\right\rfloor
+2\sum_{m=1}^{\left \lfloor{\lambda}\right \rfloor }\left\lfloor H(m,\lambda) + \frac{1}{4}\right\rfloor
\]
and
\[
\mathscr{P}_2^N(\lambda) = \left\lfloor\dfrac{\lambda}{\pi}+\dfrac{3}{4}\right\rfloor
+ 2\sum_{m=1}^{\left \lfloor{\lambda}\right \rfloor }\left\lfloor H(m,\lambda) + \frac{3}{4}\right\rfloor,
\]
where
\[
 H(m,\lambda) = \frac{1}{\pi} \left( \sqrt{\lambda^2 - m^2} -m \arccos(\frac{m}{\lambda})\right).
\]
From~\cite[Theorem~2.3]{flps} we have $\mathscr{N}_{\mathbb{D}}^D(\lambda) \leq \mathscr{P}_2^D(\lambda)$ and
$\mathscr{P}_2^N(\lambda) \leq \mathscr{N}_{\mathbb{D}}^N(\lambda)$, yielding
\begin{equation}\label{ineqcount}
\begin{array}{lll}
\mathscr{N}_{\mathbb{D}}^N(\lambda) - \mathscr{N}_{\mathbb{D}}^D(\lambda) & \geq & \mathscr{P}_2^N(\lambda) -\mathscr{P}_2^D(\lambda)\eqskip
& = & \left\lfloor\dfrac{\lambda}{\pi}+\dfrac{3}{4}\right\rfloor - \left\lfloor\dfrac{\lambda}{\pi}+\dfrac{1}{4}\right\rfloor
+ 2\dsum_{m=1}^{\left\lfloor\lambda\right\rfloor}\left( \left\lfloor H(m,\lambda) + \frac{3}{4}\right\rfloor
- \left\lfloor H(m,\lambda) + \frac{1}{4}\right\rfloor\right)\eqskip
& \geq & 2\dsum_{m=1}^{\left\lfloor\lambda\right\rfloor}\left( \left\lfloor H(m,\lambda) + \frac{3}{4}\right\rfloor
- \left\lfloor H(m,\lambda) + \frac{1}{4}\right\rfloor\right).
\end{array}
\end{equation}
Denoting this last expression by $\mathscr{S}(\lambda)$, we have
\[
\mathscr{S}(\lambda) = 2\#\left\{ m : 1\leq m \leq \lfloor\lambda\rfloor \wedge
\left\{H(m,\lambda)\right\} \in \left[ \frac{1}{4}, \frac{3}{4} \right) \right\},
\]
where $\left\{ x \right\} = x - \left\lfloor x \right\rfloor $ is the fractional part of $x$. We thus want to obtain a lower bound on the
number of integer points in $\left[1,\left\lfloor\lambda\right\rfloor\right]$ for which the fractional part of $H(m,\lambda)$ lies in the
interval $[1/4,3/4)$. Considering $H$ as a function of the (now continuous) variable $m$, $H(\cdot,\lambda) : [0,\lambda]\to [0,\lambda/\pi]$,
we have
\[
 \fr{\partial H(m,\lambda)}{\partial m} = -\fr{1}{\pi} \arccos\left(\fr{m}{\lambda} \right).
\]
Hence $H$ is a (strictly) decreasing function of $m$, with derivative between $-1/2$ and $0$. Given $x_{0} \in [0,\lambda)$ such that $H(x_{0},
\lambda) = q+3/4$ for some non-negative integer $q$, we thus have $H(x_{0}+1,\lambda) > q+1/4$ and there exists at least one integer in
$(x_{0},x_{0}+1]$. Thus $\#\{ m : 1\leq m \leq \lfloor\lambda\rfloor \wedge \left\{H(m,\lambda)\right\} \in \left[ \frac{1}{4}, \frac{3}{4} \right)\}$
is larger than or equal to the number of bands of the form $[q+1/4, q+3/4)$, where $q$ is some non-negative integer. This last quantity is given by
\[
\# \left\{ q \in \mathbb{N}_0: \exists \text{ } x_0 \in [0, \lambda) \text{ such that } H(x_0, \lambda) = q+\frac{3}{4} \right\} = \left\lfloor H(0,\lambda) + \fr{1}{4} \right\rfloor =  \left\lfloor \fr{\lambda}{\pi} + \fr{1}{4} \right\rfloor,
\]
and we obtain
\[
\mathscr{S}(\lambda) \geq 2 \left\lfloor \fr{\lambda}{\pi} + \fr{1}{4} \right\rfloor.
\]
Replacing this in~\eqref{ineqcount} with $\lambda = \sqrt{\lambda_{k}}$ yields
\[
\begin{array}{lll}
\mathscr{N}_{\mathbb{D}}^N(\sqrt{\lambda_{k}}) & > & \mathscr{N}_{\mathbb{D}}^D(\sqrt{\lambda_{k}}) + \mathscr{S}(\sqrt{\lambda_{k}})\eqskip
& > & k + 2 \left\lfloor \fr{\sqrt{\lambda_{k}}}{\pi} + \fr{1}{4} \right\rfloor\eqskip
& \geq & k + 2\left\lfloor 2\fr{\sqrt{k}}{\pi} + \fr{1}{4} \right\rfloor,
\end{array}
\]
where in the last inequality we used the fact that the disk satisfies P\'{o}lya's conjecture~\cite{flps}.
\end{proof}

The coefficient of the term in $\sqrt{k}$ in the above result is $4/\pi\approx 1.2732$. By improving the lower bound
on the number of crossings of the graph of $H(\cdot,\lambda)$ with the bands of width $1/2$
described above at integer abscissas, it is possible to improve this value and take it up to approximately $1.59092$ -- see the remark below,
where we then compare this with what may be obtained numerically.
\begin{thm}\label{disk_thm} We have
 \[
\lambda_k \geq \mu_{k+p(k)}
\]
for all positive integer $k$, where
\[
 p(k) = 2\dsum_{n=1}^{\infty} \left\lfloor \fr{2}{\pi} \Big( \sin \left(\frac{\pi}{2n}\right)- \frac{\pi}{2n}
 \cos\left(\frac{\pi}{2n}\right) \Big)\sqrt{k} + \frac{1}{4} \right\rfloor.
\]
\end{thm}
\begin{proof}
 We proceed in the same way as in the proof of Theorem~\ref{disk1}, now improving the estimate of the function $\mathscr{S}(\lambda)$
 defined in that proof. To do this, instead of using the fact that the derivative of $H$ with respect to $m$ is negative and larger
 than $-1/2$, we now observe that this derivative takes on the value of $-1/(2n)$ for positive integer $n$ at points where
\[
 -\fr{1}{\pi} \arccos\left( \fr{m}{\lambda}\right) = -\fr{1}{2n},
\]
 which has solutions of the form $m=\lambda \cos\left(\fr{\pi}{2n}\right)$. Due to the convexity of $H$ as a function of $m$, we
 have
 \[
 \left( 0 > \right) \fr{\partial H(m,\lambda)}{\partial m} \geq -\frac{1}{2n} \mbox{ for all } m \in\left[
 \lambda \cos\left(\fr{\pi}{2n}\right),\lambda\right).
 \]
Using an argument similar to that above, given $x_{0} \in [\lambda \cos(\fr{\pi}{2n}),\lambda)$ such that $H(x_{0},\lambda) = q+3/4$ for some non-negative integer $q$, we then have $H(x_{0}+n,\lambda) > q+1/4$. Since there now exist at least $n$ integers in $(x_{0},x_{0}+n]$,
the set $\{ m : \lambda \cos(\fr{\pi}{2n})\leq m \leq \lfloor\lambda\rfloor \wedge \left\{H(m,\lambda)\right\}\in\left[ \frac{1}{4},\frac{3}{4} \right)\}$ 
has at least $n$ times the number of bands of the form $[q+1/4, q+3/4)$ elements, where $q$ is some non-negative integer, provided the existence of $x_{0}$ in the stated conditions. This value is now given by
\[
\# \left\{ q \in \mathbb{N}_0: \exists \text{ } x_0 \in \left[\lambda \cos\left(\fr{\pi}{2n}\right),\lambda\right) \text{ such that } H(x_0, \lambda) = q+\frac{3}{4} \right\} = \left\lfloor H\left( \lambda \cos\left(\fr{\pi}{2n}\right),\lambda\right) + \fr{1}{4} \right\rfloor,
\]
Since
 \[
  H\left( \lambda \cos\left(\fr{\pi}{2n}\right),\lambda\right) = \fr{\lambda}{\pi}
  \Big( \sin \left(\frac{\pi}{2n}\right)- \fr{\pi}{2n} \cos\left(\fr{\pi}{2n}\right) \Big),
 \]
and keeping in mind that $n-1$ of the integers in $(x_{0},x_{0}+n]$ have already been accounted for for smaller values of $n$, for a given
 $\lambda$, we have
 \[
   \mathscr{S}(\lambda) \geq 2\dsum_{n=1}^{\infty} \left\lfloor \fr{\lambda}{\pi}
   \Big( \sin \left(\frac{\pi}{2n}\right)- \frac{\pi}{2n} \cos\left(\frac{\pi}{2n} \right) \Big)  + \frac{1}{4} \right\rfloor.
 \]
Note that this is, in fact, a finite sum, as for large enough $n$ the term inside the floor function becomes
smaller than one.

% If $\left\{\lfloor H(m,\lambda)\right\rfloor\}\geq 1/4$, then

 If we now take $\lambda = \sqrt{\lambda_{k}}$ as before, we obtain the lower bound
 \[
 \begin{array}{lll}
  \mathscr{S}(\lambda) & \geq & 2\dsum_{n=1}^{\infty} \left\lfloor \fr{\sqrt{\lambda_{k}}}{\pi}
   \Big( \sin \left(\frac{\pi}{2n}\right)- \frac{\pi}{2n} \cos\left(\frac{\pi}{2n}\right) \Big)  + \frac{1}{4} \right\rfloor\eqskip
   & \geq & 2\dsum_{n=1}^{\infty} \left\lfloor 2\Big( \sin \left(\frac{\pi}{2n}\right)- \frac{\pi}{2n} \cos\left(\frac{\pi}{2n}\right)
    \Big) \sqrt{k}  + \frac{1}{4}\right\rfloor,
 \end{array}
 \]
 where the second step comes from applying P\'{o}lya's inequality for the disk.
\end{proof}
\begin{remark}
 By noting that the expression for $p(k)$ given above is, in fact, a finite sum for each $k$, as the term inside the floor function
 gets smaller than one for large enough $n$, we may write this as
 \[
  p(k) = 2\dsum_{n=1}^{b_{k}} \left\lfloor \fr{2}{\pi} \Big( \sin \left(\frac{\pi}{2n}\right)- \frac{\pi}{2n}
  \cos\left(\frac{\pi}{2n}\right) \Big) \sqrt{k} + \fr{1}{4} \right\rfloor
 \]
 where, for each $k$, $b_{k}$ denotes the smallest integer for which this happens. Since the sequnce multiplying $\sqrt{k}$ is
 decreasing and of order $n^{-3}$, we obtain that $b_{k}$ is of order $k^{1/6}$ and by a simple calculation we have
 \[
  p(k) \approx 2\sum_{n=1}^{\infty} \fr{2}{\pi} \Big( \sin \left(\frac{\pi}{2n}\right)- \frac{\pi}{2n}
  \cos\left(\frac{\pi}{2n}\right) \Big)\sqrt{k} \approx 1.59092 \sqrt{k}
 \]
 as $k$ goes to infinity.
\end{remark}

We end this section by presenting some numerical evaluations of the difference between the Dirichlet and Neumann eigenvalues for
$p(k) = \lfloor C\sqrt{k}\rfloor$ for two different values of $C$ (Figure~\ref{fig_diff2}). The first, with $C=2$, illustrates the
issues raised in Remark~\ref{isoperim_manif} concerning whether or not the disk satisfies Theorem~\ref{thmpasympt} in dimension two.
The second, for $C\approx 1.665221$ points in the direction that the value of $C\approx 1.59092$ obtained in Theorem~\ref{disk_thm}
above is not too far from the optimal constant for the disk, as in this case no negative values for the difference were found.
\begin{figure}[h]
        \centering
        \includegraphics[height=0.23\textheight]{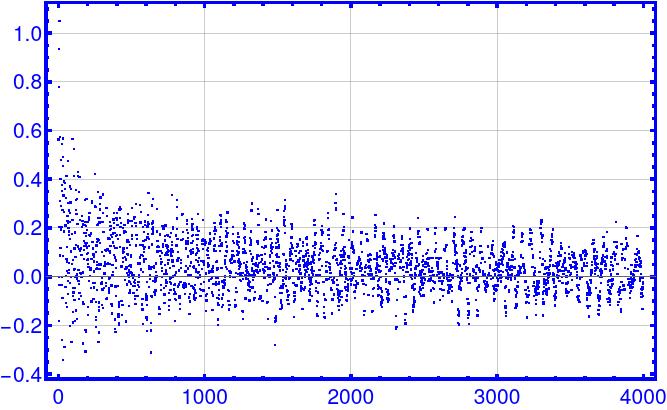}\hspace*{7mm}
        \includegraphics[height=0.23\textheight]{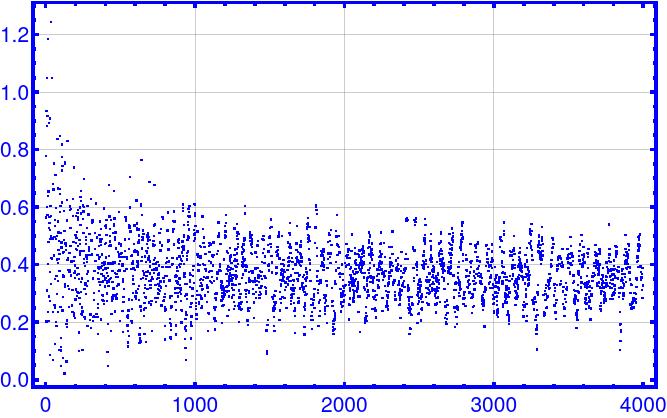}
    \caption{Graphs of the difference $\lambda_{k}(\mathbb{D})-\mu_{k+p(k)}(\mathbb{D})$ with $p(k) = \lfloor C\sqrt{k}\rfloor$,
    for $C=2$ (left) and $C\approx 1.665221$ (right)\label{fig_diff2}}
\end{figure}

We note that the results given in~\cite[Theorem 3.1]{frei} may be applied to the Neumann problem to yield
bounds for the first eigenvalue for each value of $l$, namely,
\[
\begin{array}{l}
\fr{d + 4 l + 4 -\sqrt {(d + 4)^2 + 32 l}} {2 (d + 2 l)}j_{\frac {d}{2} + l - 1, 1}^2 \leq \mu_{d,l}\eqskip
\hspace*{15mm}
\leq
\frac {j_ {\frac {d} {2} + l - 1, 1}^2} {2\left (j_ {\frac {d} {2} + l - 1, 1}^2 - 2\left (d + 2l \right) \right)}
\left[j_ {\frac {d} {2} + l - 1, 1}^2 + \left (d + 2l \right)\left (l - \sqrt {\frac
{j_ {\frac {d} {2} + l - 1, 1}^4} {\left (d + 2l \right)^2} +
      l\left (-\frac {j_ {\frac {d} {2} + l - 1, 1}^2} {\frac {d} {2} + l} + l +
          8 \right)} \right)\right].
\end{array}
\]
Using these bounds allows us to locate with precision the first Neumann eigenvalue corresponding to each harmonic
polynomial (see~\cite{frei} for a discussion of the sharpness of the two bounds), allowing us to avoid intervals where the Bessel
functions in question are too close to zero and numerical algorithms may produce spurious solutions.

\subsection{General planar domains\label{generalplanar}} The application of Bernoulli's inequality in the proof of
Theorem~\ref{thmpasympt} limits the powers of $n$ appearing in that result to $k^{1-1/n}$. By considering
the two-dimensional case we can easily do that calculation explicitly with all the terms and,
if we then keep one of the terms with the perimeter and the area, recover an improved inequality
of the type proved above for rectangles. This will, of course, still be valid only for sufficiently
large $k$ and, in fact we already know from the considerations about the disk in the previous section
that the resulting expression cannot hold for all $k$. In fact, it will be necessary to introduce
a coefficient in the $\sqrt{k}$ term for the proof to work, although it is not completely clear at
this stage whetther this is just a technical issue or has a deeper significance.
\begin{thm}\label{gen2d}
 Let $\Omega$ be a bounded planar domain with area and perimeter $A$ and $P$, respectively, and
 satisfying the non-periodicity condition. Then, for all $\alpha$ in $(0,1)$, there exists
 $k^{*} = k^{*}(\alpha,\Omega)$ such that
 \[
  \lambda_{k} \geq \mu_{k+\left\lfloor \fr{\alpha P}{\sqrt{\pi A}}\sqrt{k}\right\rfloor+1}, \mbox{ for all } k\geq k^{*}.
 \]
\end{thm}
\begin{proof}
 The proof proceeds in the same way as that of Theorem~\ref{thmpasympt} to obtain
 \[
  \lambda_{k} - \mu_{k+p} = r_{D}(k) - r_{N}(k+p)+ \fr{2\sqrt{\pi}}{A}
  \underbrace{\left[\left(\sqrt{k+p}+\sqrt{k}\right)\fr{P}{\sqrt{A}}-2p\sqrt{\pi}\right]}_{\text{g(k,p)}}.
 \]
The function $g(k,p)$ will be non-negative if
 \[
  (0 <)\; p \leq \fr{P^2}{4\pi A}+\fr{P}{\sqrt{\pi A}}\sqrt{k}.
 \]
From the two-dimensional isoperimetric inequality we have that the right-hand side above satisfies
\begin{equation}\label{twoisoperimetric}
 \fr{P^2}{4\pi A}+\fr{P}{\sqrt{\pi A}}\sqrt{k} \geq 1 + \fr{P}{\sqrt{\pi A}}\sqrt{k}
\end{equation}
and for $\alpha$ in $(0,1)$ we take $p(k)=\left\lfloor \fr{\alpha P}{\sqrt{\pi A}}\sqrt{k}
\right\rfloor+1\leq \fr{ P}{\sqrt{\pi A}}\sqrt{k}+\fr{P^2}{4\pi A}$,
yielding $g\left(k,\left\lfloor \fr{P}{\sqrt{\pi A}}\sqrt{k}\right\rfloor+1\right)\geq0$.  On the other hand,
for given $A$ and $P$, and sufficiently large values of $k$, $g(k,p)$ is a decreasing function of $p$.
Writing $c = P/(2\sqrt{\pi A})$, we have $p(k) \leq \fr{\alpha P}{\sqrt{\pi A}}\sqrt{k} +1 = 2\alpha c\sqrt{k}+ 1$,  and
\[
g(k,p) \geq 2\sqrt{\pi} \left[\left( \sqrt{k+2\alpha c\sqrt{k}+1}+\sqrt{k} \right)c - (2\alpha c \sqrt{k}+1)\right] = 4\sqrt{\pi}(1-\alpha)\sqrt{k} +
\bo(1),
\]
as $k$ goes to infinity.
\end{proof}
\begin{remark}
By also bounding the second occurrence of the perimeter and the area in~\eqref{twoisoperimetric} using the isoperimetric
inequality, it is possible to obtain a result with $p(k) = \left\lfloor2\sqrt{k}\right\rfloor+1$. To see this,
proceed as above and note that we may now write
\[
 g(k,2\sqrt{k}+1) = \left( \fr{P}{\sqrt{A}}-2\sqrt{\pi}\right)\left(2\sqrt{k}+1\right),
\]
thus obtaining $\lambda_k \geq \mu_{k+\left\lfloor2\sqrt{k}\right\rfloor+1}$, for sufficiently large $k$, except possibly
in the case of the disk. In fact, for the disk this is not true when $k$ is one either, as $\lambda_{1}(D)<\mu_{4}(D)$,
\end{remark}

\subsection{The sphere $\mathbb{S}^2$\label{sec-sphere}} Proceeding in the same way as above for planar domains,
it is possible to derive a similar result for domains on the sphere that also satisfy the non-periodicity
condition. The resulting expressions are now more involved, mirroring the version of the isoperimetric inequality
on $\mathbb{S}^{2}$, namely,~\cite{levy}
\[
 P^{2} \geq 4\pi A -A^2.
\]

\begin{thm}\label{genS2}
 Let $\Omega\subsetneq \mathbb{S}^2$ be a domain satisfying the non-periodicity condition and which
 is not a geodesic disk. Then there exists $k^{*} = k^{*}(\Omega)$ such that
 \[
  \lambda_{k} \geq \mu_{k+p(k)}, \mbox{ for all } k\geq k^{*},
 \]
 where $p$ is given by
 \[
  p(k) = \left\lfloor 1- \fr{A}{4\pi} + 2 \sqrt{1-\fr{A}{4\pi}}\sqrt{k}\right\rfloor.
 \]
\end{thm}
\begin{proof}
The proof proceeds in the same way as that for Theorem~\ref{gen2d}.
\end{proof}
As mentioned in the Introduction, geodesic disks on $\mathbb{S}^{2}$ with radius larger than or equal to $\pi/2$ do not
satisfy the non-periodicity condition. On the other hand, those with radius smaller than $\pi/2$ do, and again we would
expect these to satisfy the above inequality.

\section*{Acknowledgements} It is a pleasure to acknowledge several exchanges with Mark Ashbaugh concerning this
problem and, in particular, for having mentioned~\cite{levi} and other relevant articles. We are also indebted to Pedro
Antunes for several conversations concerning the numerical evaluation of eigenvalues of balls and annular sectors.
This work was partially supported by the Funda\c{c}\~ao para a Ci\^encia e a Tecnologia, Portugal, via the research centre
GFM, reference UID/00208/2023.

\end{document}